\documentclass[reqno]{amsart}
\usepackage{amsmath,amssymb}
\usepackage[mathscr]{euscript}
\usepackage{graphicx}
\usepackage{enumitem}
\usepackage[english]{babel}
\usepackage{hyperref}

\usepackage[small]{caption}

\usepackage{tikz}

\makeatletter
\@addtoreset{equation}{section}
\makeatother

\newtheorem{theorem}{Theorem}[section]
\newtheorem{lemma}[theorem]{Lemma}

\newtheorem{remark}[theorem]{Remark}

\numberwithin{equation}{section}

\newcommand{\mc}[1]{{\mathcal #1}}
\newcommand{\ms}[1]{{\mathscr #1}}

\newcommand{\bb}[1]{{\mathbb #1}}

\newcounter{as}[section]

\begin{document}

\title{Condensation of the invariant measures of the supercritical zero range processes}

\author{Tiecheng Xu}

\address{\noindent IME-USP, Rua do Mat\~ao 1010, CEP 05508-090, S\~ao Paulo, Brazil.
  \newline e-mail: \rm \texttt{tc.xu@ime.usp.br} }

\begin{abstract}
 For $\alpha\geq 1$, let $g:\bb N\to\bb R_+$ be given by $g(0)=0$, $g(1)=1$, $g(k)=(k/k-1)^\alpha$, $k\geq 2$. Consider the symmetric nearest neighbour zero range process on  the discrete torus $\bb T_L$ in which a particle jumps from a site, occupied by $k$ particles, to one of its neighbors with rate $g(k)$. Armend\'ariz and Loulakis\cite{al09} proved a strong form of the equivalence of ensembles for the invariant measure of the supercritical zero range process when $\alpha>2$. We generalize their result to all $\alpha\geq 1$.  
\end{abstract}


\maketitle

\section{Introduction}
\label{sec1}

The zero range process is one of the most classical models of interacting particles systems. 
It was firstly introduced into the mathematical literature as an example of interacting Markov process at 1970 in \cite{s70}. 
Since then this model has received sustained attention, 
and plenty of mathematical achievements have been made
 including existence theorems\cite{l73}, invariant measures\cite{a82, al09},
  hydrodynamic limit\cite{kl99, g15},  metastability\cite{bl12,agl17}, etc. 

In the general setting the zero range process is a model in which many indistinguishable particles occupy sites on a lattice. 
Each site of the lattice may contain some number of particles and these particles jump between neighbouring sites with a rate $g(\cdot)$ that depends on the number of particles at the site of departure.  
With different jump rates and lattice on which the process is defined, the model may present different phenomena.  In this paper we focus on the invariant measure of the  zero range process on the one-dimensional discrete torus with $L$ sites.

The critical density, denoted by $\rho_c$, plays an important role in the study of invariant measure. Consider the nearest neighbour symmetric zero range process on $\bb T_L$ with $N$ particles and denote its invariant measure by $\mu_{N,L}$. In the subcritical regime, i.e. $N/L\to\rho<\rho_c$, we have the well known equivalence of ensembles: the marginals of the canonical measure $\mu_{N,L}$ converges to the marginals of the grand canonical product measure with mean density $\rho$.

In the supercritical regime, if the function $g(\cdot)$ decreases fast enough, the invariant measure of the process concentrates on configurations where a large portion of the total particles stays at a single site. This is called the condensation phenomenon. Ferrari, Landim and Sisko studied this condensation phenomenon of the invariant measure in \cite{fls07}, where they consider the zero range process on a fixed finite set while the number of particles $N$ grows to infinity. Grosskinsky, Sch\"utz, Spohn\cite{gss03} and then Armendariz and Loulakis \cite{al09} generalized that result to the process on an increasing domain $\bb T_L$, proving that if $N/L\to\rho>\rho_c$, then the distribution of the particles outside the condensate converges to the grand canonical distribution with critical density $\rho_c$ in different senses of convergence.

Choose $g(\cdot)$ to be the function defined in \eqref{g}. Both results of \cite{fls07} and \cite{al09} apply to the zero range process with $\alpha>2$, where $\alpha$ is the parameter in the definition of $g$. There are also some works investigating the dynamical aspects of the condensation phenomenon, in other words, the evolution of the location with the majority of particles. This problem for the zero range process  on a fixed finite set was solved by Beltr\'an and Landim\cite{bl12}, surprisingly for not just for $\alpha>2$ but for all $\alpha>1$.The problem in the case $\alpha=1$ was recently solved by Landim, Marcondes and Seo\cite{lms20}.  Armend\'ariz,  Grosskinsky and Loulakis\cite{agl17} extended the result of \cite{bl12} to the supercritical zero range process on $\bb T_L$ for $\alpha>20$. We believe that the condensation phenomenon for the zero range process on $\bb T_L$ should occur for all $\alpha\geq1$, which is the same as the zero range process on a fixed finite set, as long as the number of particles $N$ is sufficiently large compared to the number of sites $L$. 

Motivated by the reason above, our work extends  the result\cite{al09}  of Armend\'ariz and Loulakis to all $\alpha\geq 1$. More precisely we give sufficient conditions for which $N$ and $L$ need to satisfy according to the value of $\alpha$ to observe the condensation phenomenon,  and find proper product measures to approximate the distribution of particles in the remaining sites after removing the site with the most particles. The product measure depends on $N$ and $L$, which is not the same as the case $\alpha>2$ in \cite{al09}. However, since the critical density $\rho_c$ is finite when $\alpha>2$, our result would imply the result of \cite{al09}.

\section{Model and Results}\label{sec2}
\subsection{Basic notation}
Consider two sequences of non-negative numbers $\{a_L: L\in \bb Z_+\}$ and $\{b_L: L\in \bb Z_+\}$. We write $a_L=O(b_L)$ if there exists a positive constant $C>0$ and a postive integer $L_0$ such that $a_L\leq C b_L$ for all $L\geq L_0$, and write $a_L=o(b_L)$ if $\lim_{L\to\infty}a_L/b_L=0$. Sometimes we also use $a_L\ll b_L$ or $b_L\gg a_L$ to represent $a_L=o(b_L)$. We use notation $a_L=\Theta(b_L)$ if there exist positive constants $c,C$ and a postive integer $L_0$ such that $c b_L\leq a_L\leq C b_L$ for all $L\geq L_0$
\subsection{Zero range process on $\bb T_L$}
For each positive integer $L>0$, denote by $\bb T_L$ the one-dimensional discrete torus with $L$ sites:
$$\bb T_L\,=\, \bb Z/L\bb Z\,=\,\{1,2,\cdots, L\}.$$
For each $\alpha\in \bb R$, define a function $g:\bb N\to \bb R$:
\begin{equation}\label{g}
g(0)=0, \quad g(1)=1, \quad \text{and}\quad g(n)=\frac{n^\alpha}{(n-1)^\alpha}, \quad n\geq 2.
\end{equation}
Define $a(n)=\prod_{i=1}^n g(i)=n^\alpha$ for $n\geq 1$ and set $a(0)=1$. 

 Let $X_L=\bb N^{\bb T_L}$. Consider the nearest neighbour symmetric zero range process $\{\eta^{L}(t): t\geq 0\}$ with state space $X_{L}$ whose generator $\ms L_L$ acts on functions $F:X_L\to \bb R$ as 
\begin{equation}\label{gen}
(\ms L_L F)(\eta)\,=\,\sum_{\substack{x\in \bb T_L\\ |y-x|=1}}\frac{1}{2}\,g(\eta_x)\,\left\{F(\sigma^{x,y}\eta)-F(\eta)\right\}.
\end{equation}
In the above equation \eqref{gen}, if $\eta_x>0$, $\sigma^{x,y}\eta$ is the configuration obtained from $\eta$ by moving one particle from site $x$ to $y$:
\begin{equation} 
(\sigma^{x,y}\eta)_z=
\begin{cases}
\eta_x-1, & \text{for} \quad z=x,\\
\eta_y+1, & \text {for} \quad z=y,\\
\eta_z, & \text{otherwise}.
\end{cases}
\end{equation}
 For positive integers $N,L \geq 1$, define the set of configurations with $N$ particles staying at $\bb T_L$ by:
\begin{equation*}
E_{N,L}\,=\,\{\eta\in\bb N^{\bb T_L}:\sum_{x\in S_0} \eta_x \,=\, N\},
\end{equation*}
where $\bb N\,=\,\{0,1,2,\cdots\}$. Then $X_L$ is the union of disjoint sets $E_{N,L}$ over all $N\geq 1$. Note that this dynamics conserves the number of particles. Therefore we could restrict
the process $\{\eta^{L}(t): t\geq 0\}$ on the hyperplane $E_{N,L}$ and denoted this restricted process by $\{\eta^{N,L}(t): t\geq 0\}$.

\subsection{The canonical measure and grand canonical measure}
The Markov process $\{\eta^{N,L}(t): t\geq 0\}$ is irreducible and reversible with respect to its unique invariant measure $\mu_{N,L}$ given by:
\begin{equation}
\mu_{N,L}(\eta)\,=\,\frac{1}{Z_{N,L}}\prod_{x\in \bb T_L}\frac{1}{a(\eta_x)},
\end{equation}
where $Z_{N,L}$ is the normalizing constant 
\begin{equation}\label{defZNL}
Z_{N,L}\,=\,\sum_{\eta\in E_{N,L}}\prod_{x\in \bb T_L}\frac{1}{a(\eta_x)}.
\end{equation}
The measure $\mu_{N,L}$ is the so-called canonical measure.

Define a function $Z:\bb R_+\to\bb R_+$ by
$$Z(\varphi)\,=\, \sum_{k\geq 0}\frac{\varphi^k}{a(k)}.$$
Let $\varphi_c$ be the radius of convergence of function $Z$, then a simple computation shows that $\varphi_c=1$ for all $\alpha\in\bb R$. It is easy to check that $Z(\varphi_c)<\infty$ if and only if $\alpha>1$. For each integer $N\geq 1$ and $\varphi\in (0,\infty)$, in order to deal with the case $\alpha\in[1,2]$ later,  we also need to define the truncation of $Z(\varphi)$ by
$$Z_N(\varphi)\,:=\, \sum_{k=0}^N\frac{\varphi^k}{a(k)}.$$
For each $\alpha\in\bb R$, let $D_\alpha$ be the set of positive $\varphi$ such that $Z(\varphi)<\infty$. Given any $\varphi\in D_\alpha$, we can define the following probability measure on $\bb N$:
$$\nu_{\varphi}[k]\,=\,\frac{1}{Z(\varphi)}\frac{\varphi^k}{a(k)}, \quad k\geq 0.$$
Let $\nu_\varphi^L$ be the product measure on $X_L$ with marginals $\nu_\varphi$. It is well known that $\nu_\varphi^L$ is invariant for the process $\{\eta^L(t):t\geq 0\}$. The measure $\nu_\varphi^L$ is called the grand canonical measure. The following identity relates measures $\mu_{N,L}$ and $\nu_{\varphi}^L$:
\begin{equation*}
\mu_{N,L}[\eta]\,=\,\nu_{\varphi}^L\left[\,\eta\,\Big\vert\sum_{x\in \bb T_L}\eta_x=N\right], \quad \eta\in E_{N,L}.
\end{equation*}
Note that the right hand side of the above equation does not depend on $\varphi$.
For each integer $N\geq 1$ and $\varphi\in(0,\infty)$, we also define the measure $\nu_{\varphi, N}$ by
$$\nu_{\varphi,N}[k]\,=\,\frac{1}{Z_N(\varphi)}\frac{\varphi^k}{a(k)}, \quad 0\leq k\leq N$$
and define the measure $\nu_{\varphi,N}^L$ to be the product measure on $X_L$ with marginals $\nu_{\varphi,N}.$

Given $\varphi\in D_\alpha$, the expected number of particles per site under measure $\nu_{\varphi}^L$ is given by
$$\rho(\varphi)\,=\,\frac{1}{Z(\varphi)}\sum_{k\geq 0}\frac{k\varphi^k}{a(k)}.$$
It is easy to verify that $\rho$ is a strictly increasing function of $\varphi$. Define the critical density $\rho_c$ by $\rho_c:=\lim_{\varphi\uparrow\varphi_c} \rho(\varphi)$. One can check that $\rho_c<\infty$ if and only if $\alpha>2$.

\subsection{Condensation} For a configuration $\eta\in X_L$, define 
$$M_L(\eta)\,:=\,\max_{x\in \bb T_L}\eta_x$$
and let $m_L(\eta)$ be the location of the maximum. In the case there exist more than one maximum, choose randomly with equal probability one of the sites of maximum to be $m_L(\eta)$. Define the map $T:X_L\to X_{L-1}$ that removes the site of maximum. For example, given a configuration $\eta\in X_L$ and suppose $m_L(\eta)=k$, then $T\eta=(\eta_1,\cdots,\eta_{k-1},\eta_{k+1},\cdots,\eta_L)$.
The following result is obtained by I.Armend\'ariz and M.Loulakis\cite{al09}, even though they consider a  function $g(\cdot)$ slightly different from ours.
\begin{theorem}\label{thmIM}
Let $\mc F_L$ be the $\sigma$-algebra generated by $\eta_1,\cdots,\eta_L$. Assume that $\alpha>2$ and
$$\lim_{N,L\to\infty}\frac{N}{L}\,=\,\rho>\rho_c.$$
Then
$$\lim_{\substack{N,L\to\infty\\N/L\to\rho}}\sup_{A\in\mc F_{L-1}}\left|\mu_{N,L}[T^{-1}A]\,-\,\nu_{\varphi_c}^{L-1}[A]\right|\,=\,0.$$
\end{theorem}
This theorem indicates that for $\alpha>2$, when the density is supercritical a condensation phenomen  emerges. The site of maximum contains around $(\rho-\rho_c)L$ particles while about only $\rho_c$ particles stays at each of the rest sites.

We would like to investigate if the condensation phenomenon occurs for $\alpha\leq 2$, in the sense that, the site of maximum contains much more particles than any other site. Our result confirms that the condensation phenomenon still occurs for $\alpha\geq1$, as long as there are enough particles in the system.

We define the ``critical density" for all $\alpha\geq1$ like the role of $\rho_c$ in the case $\alpha>2$:
$$\rho_{c,N}\,:=\, E_{\nu_{\varphi_c, N}}[\eta_x]\,=\,\frac{\sum_{k=0}^N\frac{k}{a(k)}}{Z_N(\varphi_c)}.$$
Note that if $\alpha>2$, then $\rho_{c,N}\to\rho_c$ as $N\to\infty$. A straightforward computation gives the order of $\rho_{c,N}$:
\begin{equation}\label{rhocN}
\rho_{c,N}\,=\,
\begin{cases}
\Theta(1)  &\text{if} \quad \alpha>2\\
\Theta(\log N) &\text{if} \quad \alpha=2\\
\Theta(N^{2-\alpha}), &\text{if} \quad \alpha\in (1,2)\\
\Theta(\frac{N}{\log N}), &\text{if} \quad \alpha=1
\end{cases}
\end{equation}

Now we are ready to state our main result.
\begin{theorem}\label{conden}
Let $\mc F_L$ be the $\sigma$-algebra generated by $\eta_1,\cdots,\eta_L$. Assume $\alpha\geq 1$. Consider sequences of postive integers $\{\rho_L: L\geq 1\}$ and $\{k_L: L\geq 1\}$ such that 
$$\lim_{L\to\infty}\frac{k_L}{L\rho_L}\,=\,0$$
and let 
\begin{equation}\label{NL}
N\,=\,N(L)\,=\, L\rho_L+k_L.
\end{equation}
Assume that $\rho_L$ satisfies
\begin{itemize}
\item if $\alpha>2$, $\liminf\limits_{L\to\infty}\rho_L>\rho_c$,
\item if $\alpha=2$, $\rho_L\gg \rho_{c,N}$,
\item if $\alpha\in(1,2)$, $\rho_L\gg \rho_{c,N}\log N$
\item if $\alpha=1$, $\rho_L\gg \rho_{c,N}(\log\log N)^\delta$, for some $\delta>1$.
\end{itemize}
then
$$\lim_{L\to\infty}\sup_{A\in\mc F_{L-1}}\left|\mu_{N,L}[T^{-1}A]\,-\,\nu_{\varphi_c,N}^{L-1}[A]\right|\,=\,0.$$
\end{theorem}
\begin{remark}
One can easily check that  if $\alpha>2$, 
$$\lim_{\substack{N,L\to\infty\\N/L\to\rho}}\sup_{A\in\mc F_{L-1}}\left|\nu_{\varphi_c,N}^{L-1}[A]\,-\,\nu_{\varphi_c}^{L-1}[A]\right|\,=\,0.$$
By the triangle inequality, Theorem \ref{thmIM} can deduced from this Theorem.
\end{remark}

\begin{remark}
For every $\alpha\geq 1$ and sufficiently large $N$, the site of maximum contains $\Theta(N)$ particles while about only $\rho_{c,N}$ particles stay at each of the rest sites. Those assumptions on $\rho_L$ imply that $\rho_{c,N}\ll N$ for all $\alpha\geq1$.  Therefore the condensation phenomenon occurs for every $\alpha\geq1$.
\end{remark}

\begin{remark}
The assumptions on $\rho_L$ in the Theorem to observe the condensation phenomenon of the invariant measure  are not optimal. Actually even for $\alpha>2$, Armend\'ariz, Grosskinsky and Loulakis\cite{agl13} showed that if $\sqrt{L}\ll N-\rho_c L\ll L$, the invariant measures of the zero range processes still exhibit a condensation phenomenon.
\end{remark}

\begin{section}{Proof of Theorem \ref{conden}}
The proof of Theorem \ref{conden} is similar to the proof of Theorem \ref{thmIM}. One of the main steps is to prove a local limit theorem for all $\alpha\geq 1$. Our proof of it is inspired by ideas of the paper \cite{d89}. Since $N$ can be thought as a function of $L$ sastisfying equation \eqref{NL}, when we consider the limit as one of them tends to infinity, it means the limit when both of them tend to infinity. All the constants $C$ in this section would not depend on $N$ or $L$, and may change from line to line. 
\begin{proof}[Proof of Theorem \ref{conden}]
Given an event $A\in \mc F_{L-1}$, define the event in $\mc F_L$:
$$(A,M_L)\,=\,\{\eta\in X_L: (\eta_1,\cdots,\eta_{L-1})\in A, \,M_L(\eta)=L\}.$$
Since $\mu_{N,L}$ is translation invariant with respect to the operator $\sigma^{x,y}$ for each pair $x,y\in\bb T_L$, we have 
\begin{equation}\label{eq1}
\begin{split}
\mu_{N,L}[T^{-1}A]\,&=\,\sum_{x=1}^L\mu_{N,L}\left[T^{-1}A\cap\{M_L=x\}\right]\\
&=\,L\mu_{N,L}[(A,M_L)]\\
&=\, L\nu_{\varphi_c, N}^L\left[(A, M_L)\Big\vert\sum_{x\in\bb T_L}\eta_x=N\right].\\
\end{split}
\end{equation}
The last indentity follows from the explicit expression of the invariant measure $\mu_{N,L}$.

Consider a sequence $a_L$ such that 
\begin{equation}
a_L\,=\,
\begin{cases}
\sqrt{L}, & \text{if} \quad \alpha>3\\
\sqrt{L\log N}, & \text {if} \quad \alpha=3\\
\sqrt{LN^{3-\alpha}}, & \text{if} \quad \alpha\in(1,3)\\
\frac{\sqrt{LN^2}}{\sqrt{\log N}}, &\text{if} \quad \alpha=1
\end{cases}
\end{equation}
and a sequence $C_L$ such that $a_L\ll C_L\ll N$. Let $t_L= N-(L-1)\rho_{c,N}-C_L$ and $t_L^+= N-(L-1)\rho_{c,N}+C_L$. From the assumption on $\rho_L$ and the order of $\rho_{c,N}$ given in \eqref{rhocN}, $t_L$ and $t_L^+$ are of order $\Theta(N)$.

Define the event 
$$B_{N,L}\,=\,\left\{\eta\in X_L:\, \left|N-(L-1)\rho_{c,N}-M_L(\eta)\right|\leq C_L, \max_{1\leq x\leq L-1}\eta_x\leq t_L\right\}.$$ 
Since $\nu_{\varphi_{c,N}}^L$ is a product measure, $\nu_{\varphi_{c,N}}^L\left[(A, M_L)\cap B_{N,L}\cap\left\{ \sum_{x\in \bb T_L}\eta_x=N\right\}\right]$ is equal to
\begin{equation}\label{eqA}
\sum_{m=t_L}^{t_L^+}\nu_{\varphi_{c,N}}[m]\nu_{\varphi_{c,N}}^{L-1}\left[A\cap\left\{\max_{1\leq x\leq L-1}\eta_x\leq t_L\right\}\cap\left\{\sum_{x=1}^{L-1}\eta_x=N-m\right\}\right].
\end{equation}

For all integers $m$ such that $t_L\leq m\leq t_L^+$,  since $C_L\ll N$,
\begin{equation}\label{eqasmp}
\lim_{L\to\infty}\frac{\nu_{\varphi_{c,N}}[m]}{\nu_{\varphi_{c,N}}[N-(L-1)\rho_{c,N}]}\,=\,1.
\end{equation}

We claim that 
\begin{equation}\label{claim1}
\lim_{L\to\infty}\nu_{\varphi_{c,N}}^{L-1}\left[\max_{1\leq x\leq L-1}\eta_x\leq t_L\right]\,=\,1.
\end{equation}
Indeed, the probability at the left hand side is equal to 
$$\prod_{x=1}^{L-1}\nu_{\varphi_{c,N}}[\eta_x\leq t_L]\,=\,\left(1\,-\,\frac{\sum_{k=t_L+1}^N\frac{1}{a(k)}}{Z_N(\varphi_c)}\right)^L.$$
Therefore to prove the claim, it is sufficient to show that 
\begin{equation}\label{eqexp}
\lim_{L\to\infty} \frac{L\sum_{k=t_L+1}^N\frac{1}{a(k)}}{Z_N(\varphi_c)}\,=\,0.
\end{equation}
A simple computation shows that
\begin{equation*}
Z_N(\varphi_c)\,=\,
\begin{cases}
\Theta(1), &\text{if} \quad \alpha\in (1,\infty)\\
\Theta(\log N), &\text{if} \quad \alpha=1
\end{cases}
\end{equation*}
and
\begin{equation*}
\sum_{k=t_L+1}^N\frac{1}{a(k)}\,=\,
\begin{cases}
O\left(t_L^{-(\alpha-1)}-N^{-(\alpha-1)}\right), &\text{if} \quad \alpha\in (1,\infty)\\
O\left(\log N-\log t_L\right), &\text{if} \quad \alpha=1
\end{cases}
\end{equation*}
Equation \eqref{eqexp} follows from the these two estimates, the assumption on $\rho_L$ and 
$C_L\ll N$.

Moreover, we claim that 
\begin{equation}\label{claim2}
\lim_{L\to\infty}\nu^{L-1}_{\varphi_c, N}\left[\left|\sum_{x=1}^{L-1}\eta_x-(L-1)\rho_{c,N}\right|<C_L\right]\,=\,1.
\end{equation}
By Chebyshev's inequality, 
$$\nu^{L-1}_{\varphi_c, N}\left[\left|\sum_{x=1}^{L-1}\eta_x-(L-1)\rho_{c,N}\right|\geq C_L\right]\,\leq\,\frac{1}{C_L^2}\sum_{x=1}^{L-1}\left(\nu_{\varphi_c, N}[\eta_x^2]-\rho_{c,N}^2\right).$$
The order of $\nu_{\varphi_c, N}[\eta_x^2]$ can be computed easily:
\begin{equation}\label{den2}
\nu_{\varphi_c, N}[\eta_x^2]\,=\,
\begin{cases}
\Theta(1) & \text{if} \quad \alpha>3\\
\Theta(\log N) & \text{if} \quad \alpha=3\\
\Theta(N^{3-\alpha}) &\text{if} \quad \alpha\in (1,3)\\
\Theta(\frac{N^2}{\log N}) &\text{if} \quad \alpha=1
\end{cases}
\end{equation}
The claim follows from this estimate and that $C_L\gg a_L$.

By \eqref{eqasmp}\eqref{claim1} and \eqref{claim2}, we conclude that for any event $A\in\mc F_{L-1}$, 
$$\nu_{\varphi_{c,N}}^L\left[(A, M_L)\cap B_{N,L}\cap\left\{ \sum_{x\in \bb T_L}\eta_x=N\right\}\right]$$
is equal to 
$$\nu_{\varphi_{c,N}}\left[N-(L-1)\rho_{c,N}\right]\left(\nu^{L-1}_{\varphi_{c,N}}[A]+o(1)\right),$$
where the error is uniformly small in $A$. 
Replacing $A$ by $A\cap T(B_{N,L})$ in equation \eqref{eq1}, we obtain
\begin{equation*}
\mu_{N,L}[T^{-1}\left(A\cap T(B_{N,L})\right)]\,=\, L\nu_{\varphi_c, N}^L\left[(A, M_L)\cap B_{N,L}\vert\sum_{x\in\bb T_L}\eta_x=N\right].
\end{equation*}
Combing these two identities and Theorem \ref{lclt}, we get
\begin{equation}
\lim_{L\to\infty}\sup_{A\in\mc F_{L-1}}\left|\mu_{N,L}[T^{-1}(A\cap T(B_{N,L}))]\,-\,\nu_{\varphi_c,N}^{L-1}[A]\right|\,=\,0
\end{equation}
Choosing $A=X_L$ in particular, this equation implies 
$$\lim_{L\to\infty}\mu_{N,L}[T^{-1}( T(B_{N,L})^c)]=0,$$
where $T(B_{N,L})^c$ is the complement set of $T(B_{N,L})$ in $X_{L-1}$.
The assertion of the Theorem follows by combining the last two equations.
\end{proof}

\begin{theorem}[The Local Limit Theorem]\label{lclt}
For every $\alpha\geq 1$,
$$\lim_{L\to\infty}\frac{\nu^L_{\varphi_{c,N}}\left[\sum_{x=1}^L\eta_x=N\right]}{L\nu_{\varphi_{c,N}}\left[N-(L-1)\rho_{c,N}\right]}=1.$$
\end{theorem}
\begin{proof}
Define $B_L$ such that 
\begin{equation}
\frac{N}{B_L}\,=\,
\begin{cases}
(\log N)^2, & \text{if} \quad \alpha>2\\
\rho_L^{\frac{1}{4}}(\log N)^{\frac{3}{4}}, & \text {if} \quad \alpha=2\\
3\log N, & \text{if} \quad \alpha\in(1,2)\\
\log N(\log\log N)^{\delta/2}, &\text{if} \quad \alpha=1
\end{cases}
\end{equation}
where in the last line, $\delta$ is the one from the assumption on $\rho_L$ in Theorem \ref{conden}.
Denote by $\xi$ the number of sites which contain more than $B_L$ particles,
$$\xi\,:=\, \left|\left\{\ x\in \bb T_L: \eta_x\geq B_L\right\}\right|,$$
where $|\cdot|$ represents the cardinality of the set.
We divide the event $\left\{\sum_{x=1}^L\eta_x=N\right\}$ into three sets defined as follows:
$$E_i\,:=\, \left\{\sum_{x=1}^L\eta_x=N,\, \xi= i \right\}, \quad i=0,1$$
$$E_2\,:=\, \left\{\sum_{x=1}^L\eta_x=N,\, \xi\geq 2 \right\}.$$
Then 
$$\nu^L_{\varphi_{c,N}}\left[\sum_{x=1}^L\eta_x=N\right]\,=\,\sum_{i=0}^2\nu^L_{\varphi_{c,N}}\left[E_i\right].$$
The theorem follows easily from the next three lemmas.
\end{proof}

In the following three lemmas we will compare each $\varphi_{c,N}^L\left[E_i\right]$ with respect to 
$$L\nu_{\varphi_{c,N}}\left[N-(L-1)\rho_{c,N}\right].$$ Firstly we compute the order of the latter.
A direct computation shows that there exists a constant $C>0$ such that for sufficiently large $N$, if $\alpha>1$,
$$\nu_{\varphi_{c,N}}\left[N-(L-1)\rho_{c,N}\right]\,\geq\, \frac{C}{N^\alpha}$$
and if $\alpha=1$,
$$\nu_{\varphi_{c,N}}\left[N-(L-1)\rho_{c,N}\right]\,\geq\, \frac{C}{N\log N}.$$

\begin{lemma}\label{E0}
$$\lim_{L\to\infty}\frac{\nu^L_{\varphi_{c,N}}\left[E_0\right]}{L\nu_{\varphi_{c,N}}\left[N-(L-1)\rho_{c,N}\right]}\,=\,0$$
\end{lemma}
\begin{proof}
Define 
$$T_{L}\,=\,\sum_{x=1}^L\eta_x\mathbf{1}_{\{\eta_x<B_L\}}.$$
For every $s>0$, 
\begin{equation*}
\begin{split}
\nu^L_{\varphi_{c,N}}\left[E_0\right]&\leq \, \nu^L_{\varphi_{c,N}}\left[\sum_{x=1}^L\eta_x\geq N, \xi=0\right]\\
&=\, \nu^L_{\varphi_{c,N}}\left[\sum_{x=1}^L\eta_x\geq N, M_L(\eta)<B_L \right]\\
&\leq\,  \nu^L_{\varphi_{c,N}}\left[T_{L}\geq N\right]\,\leq\, \frac{E_{\nu^L_{\varphi_{c,N}}}\left[e^{sT_{L}}\right]}{e^{sN}}.
\end{split}
\end{equation*}
In the rest part of the proof of this lemma, we always choose $s=\frac{1}{B_L}$. In particular by the above inequality, we obtain that
\begin{equation}\label{E0eq0}
\nu^L_{\varphi_{c,N}}\left[E_0\right]\leq \frac{\left(E_{\nu_{\varphi_{c,N}}}\left[e^{B_L^{-1}\eta_x\mathbf{1}_{\{\eta_x<B_L\}}}\right]\right)^L}{e^{N/B_L}}.
\end{equation}
By a summation by parts,
\begin{equation}\label{E0eq1}
\begin{split}
&E_{\nu_{\varphi_{c,N}}}\left[e^{s\eta_x\mathbf{1}_{\{\eta_x<B_L\}}}\right]\leq\,\sum_{k=0}^{B_L}\left(1+e^{sk}-1\right)\nu_{\varphi_{c,N}}\left[k\right]\\
\leq& \sum_{k=0}^{B_L}\nu_{\varphi_{c,N}}\left[k\right]\,+\, \sum_{k=0}^{B_L}e^{sk}(e^s-1)\sum_{j=k}^{B_L} \nu_{\varphi_{c,N}}[j].
\end{split}
\end{equation}

Obviously the first term at the last line  is less than $1$ since $B_L<N$. Given any $\epsilon>0$, applying Taylor expansion to the exponential function, we have for $N$ sufficiently large,
$$e^{\frac{1}{B_L}}-1\leq (1\,+\,\epsilon)B_L^{-1}.$$
Let $R(k)=\sum_{j=k}^{B_L} \nu_{\varphi_{c,N}}[j]$ for all integers $k\geq 0$. By a summation by parts once more, we can bound $ \sum_{k=0}^{B_L}e^{sk}\sum_{j=k}^N \nu_{\varphi_{c,N}}[j]$ from above by
\begin{equation}
\sum_{k=0}^{B_L}R(k)\,+\, \sum_{k=0}^{B_L}e^{sk}(e^s-1)\sum_{l=k}^{B_L} R(l).
\end{equation}
A straightforward computation shows that 
\begin{equation*}
\sum_{k=0}^{B_L}R(k)\,\leq\,\rho_{c,N}
\end{equation*}
and for all $\alpha\geq1$,
\begin{equation*}
\sum_{k=0}^{B_L}\sum_{l=k}^{B_L}R(l)\,\leq\, C\nu_{\varphi_c, N}[\eta_x^2].
\end{equation*}
Recall that the order of $\nu_{\varphi_c, N}[\eta_x^2]$ has been given in \eqref{den2}. However we need a sharper upper bound in the case $\alpha\in[1,2]$:
\begin{equation}
\sum_{k=0}^{B_L}\sum_{l=k}^{B_L}R(l)\,=\,
\begin{cases}
O(B_L^{3-\alpha}) &\text{if} \quad \alpha\in (1,2]\\
O(\frac{B_L^2}{\log N}) &\text{if} \quad \alpha=1
\end{cases}
\end{equation}

If $\alpha>2$, the second term in the last line of \eqref{E0eq1} is bounded by
$$(1\,+\,\epsilon)s\left[\rho_{c,N}\,+\,(1+\epsilon)s\sum_{k=0}^{B_L}\sum_{l=k}^{B_L}R(l)\right]\,\leq\, \frac{(\log N)^2}{N}(\rho_c+2\epsilon),$$
for sufficiently large $N$, if $\epsilon$ is small enough. Moreover we could choose $\epsilon>0$ even smaller if necessary such that $\rho_c+3\epsilon<\liminf_{L\to\infty} \rho_L$. From \eqref{E0eq0} and the elementary inequality 
$$(1+x)^L\leq e^{Lx}, \quad x\geq 0,$$
we have for sufficiently large $L$,
$$\nu^L_{\varphi_{c,N}}\left[E_0\right]\,\leq\, \frac{exp\left\{\frac{L(\log N)^2}{N}(\rho_c+2\epsilon)\right\}}{exp\left\{(\log N)^2\right\}}\,\leq\,exp\left\{-\frac{\epsilon}{\rho_c+3\epsilon}(\log N)^2\right\}\ll \frac{L}{N^\alpha}.$$

 If $\alpha=2$, then an upper bound of the second term in the last line of \eqref{E0eq1} is 
$$Cs\left[\rho_{c,N}\,+\,Cs\sum_{k=0}^{B_L}\sum_{l=k}^{B_L}R(l)\right]\,\leq\,Cs\left[\log N+s B_L\right]\,\leq\,C\frac{\rho_L^{\frac{1}{4}}(\log N)^{\frac{7}{4}}}{N}.$$
. Recall that $\rho_L\gg \log N$, for sufficiently large $L$, 
$$\nu^L_{\varphi_{c,N}}\left[E_0\right]\,\leq\, \frac{exp\left\{C\frac{L\rho_L^{\frac{1}{4}}(\log N)^{\frac{7}{4}}}{N}\right\}}{exp\left\{\rho_L^{\frac{1}{4}}(\log N)^{\frac{3}{4}}\right\}}\ll \frac{L}{N^2}.$$

If $\alpha\in(1,2)$, the second term in the last line of \eqref{E0eq1} is less than or equal to
$$Cs\left[\rho_{c,N}\,+\,Cs\sum_{k=0}^{B_L}\sum_{l=k}^{B_L}R(l)\right]\,\leq\,Cs\left[N^{2-\alpha}+s B_L^{3-\alpha}\right]\,\leq\,CN^{1-\alpha}\log N.$$
Since $\rho_L\gg N^{2-\alpha}$, we conclude that, for sufficiently large $L$,
$$\nu^L_{\varphi_{c,N}}\left[E_0\right]\,\leq\,\frac{exp\left\{CLN^{1-\alpha}\log N\right\}}{exp\{3\log N\}}\ll\, \frac{L}{N^\alpha}.$$

If $\alpha=1$, the second term in the last line of \eqref{E0eq1} is bounded from above by 
$$Cs\left[\rho_{c,N}\,+\,Cs\sum_{k=0}^{B_L}\sum_{l=k}^{B_L}R(l)\right]\,\leq\,Cs\left[\frac{N}{\log N}+ \frac{sB_L^2}{\log N}\right]\,\leq\,C(\log\log N)^{\delta/2}.$$
From $\rho_L\gg \frac{N}{\log N}$ we can obtain that
$$\nu^L_{\varphi_{c,N}}\left[E_0\right]\,\leq\, \frac{exp\left\{CL(\log\log N)^{\delta/2}\right\}}{exp\left\{\log N(\log\log N)^{\delta/2}\right\}}\ll \frac{L}{N\log N},$$
for sufficiently large $L$.
\end{proof}

\begin{lemma}\label{E1}
$$\lim_{L\to\infty}\frac{\nu^L_{\varphi_{c,N}}\left[E_1\right]}{L\nu_{\varphi_{c,N}}\left[N-(L-1)\rho_{c,N}\right]}\,=\,1$$
\end{lemma}
\begin{proof}
Notice that $\nu^L_{\varphi_{c,N}}$ is translation invariant under the operator $\sigma^{x,y}$ for $x,y\in \bb T_L$,
$$\frac{1}{L}\,\nu^L_{\varphi_{c,N}}\left[E_1\right]\,=\,\nu^L_{\varphi_{c,N}}\left[\sum_{x=1}^L\eta_x=N,\,M_L(\eta)\geq B_L,\,\max_{1\leq x\leq L-1}\eta_x<B_L\right].$$
The right hand side of the above equation is equal to
\begin{equation}
\begin{split}
&\sum_{k=0}^{N-B_L} \nu_{\varphi_{c,N}}\left[N-k\right]\nu^{L-1}_{\varphi_{c,N}}\left[\sum_{x=1}^{L-1}\eta_x=k,\,\max_{1\leq x\leq L-1}\eta_x<B_L\right]\\
:=&\,A_1\,+\,A_2\,+\,A_3,
\end{split}
\end{equation}
where $A_1$ is the summation over $0\leq k< (L-1)\rho_{c,N}-C_L$, $A_2$ is the summation over $(L-1)\rho_{c,N}-C_L\leq k\leq (L-1)\rho_{c,N}+C_L$ and $A_3$ is the summation over $(L-1)\rho_{c,N}-C_L<k\leq N-B_L$.

Notice that for any fixed $N$, $\nu_{\varphi_{c,N}}[k]$ is a decreasing function of $k$, therefore
$$A_1\leq \nu_{\varphi_{c,N}}\left[N-(L-1)\rho_{c,N}+C_L\right]\nu^{L-1}_{\varphi_{c,N}}\left[\sum_{x=1}^{L-1}\eta_x\leq(L-1)\rho_{c,N}-C_L\right].$$
By $\eqref{claim2}$ and $C_L\ll N-(L-1)\rho_{c,N}$, we obtain that $A_1\ll \nu_{\varphi_{c,N}}\left[N-(L-1)\rho_{c,N}\right]$. 

We claim that 
$$\lim_{L\to\infty}\nu_{\varphi_{c,N}}^{L-1}\left[\max_{1\leq x\leq L-1}\eta_x\leq B_L\right]\,=\,1.$$
Similar to the proof of \eqref{claim1}, we just need to verify
\begin{equation}
\lim_{L\to\infty} \frac{L\sum_{k=B_L}^N\frac{1}{a(k)}}{Z_N(\varphi_c)}\,=\,0.
\end{equation}
This equation follows from the assumption on $\rho_L$ and the two estimates after equation \eqref{eqexp} with $B_L$ in place of $t_L$. By the claim and repeating the same estimate to equation \eqref{eqA} in the proof of Theorem \ref{conden} while choosing $A=X_{L-1}$, we get
$$A_2\,=\,\nu_{\varphi_{c,N}}\left[N-(L-1)\rho_{c,N}\right]\left(1+o(1)\right).$$

It remains to show that $A_3\ll \nu_{\varphi_{c,N}}\left[N-(L-1)\rho_{c,N}\right]$. 
Since $\nu_{\varphi_{c,N}}[\cdot]$ is decreasing,
$$A_3\leq \nu_{\varphi_{c,N}}\left[B_L\right]\nu^{L-1}_{\varphi_{c,N}}\left[\sum_{x=1}^{L-1}\eta_x\geq(L-1)\rho_{c,N}+C_L,\,\max_{1\leq x\leq L-1}\eta_x<B_L \right].$$

If $\alpha\in[1,2]$, by the spirit of the proof of Lemma \ref{E0}, there exists a small $\epsilon>0$ such that, for sufficiently large $N$, the second term at the right hand side of the above inequality is bounded by
$$\frac{exp\left\{L(1\,+\,\epsilon)B_L^{-1}\left[\rho_{c,N}\,+\,(1+\epsilon)B_L^{-1}\sum_{k=0}^{B_L}\sum_{l=k}^{B_L}R(l)\right]\right\}}{exp\{[(L-1)\rho_{c,N}+C_L]/B_L\}},$$
which is of order
\begin{equation}
\begin{cases}
o(1/N), &\text{if} \quad \alpha=2\\
O(exp\{-\sqrt{\log N}\}), &\text{if} \quad \alpha\in (1,2)\\
o(1/N), &\text{if} \quad \alpha=1
\end{cases}
\end{equation}
if we choose 
\begin{equation}
C_L\,=\,
\begin{cases}
N\left(\frac{\log N}{\rho_L}\right)^{1/5}, &\text{if} \quad \alpha=2\\
N(\log N)^{-1/2}, &\text{if} \quad \alpha\in (1,2)\\
N(\log\log N)^{-\delta/3}, &\text{if} \quad \alpha=1
\end{cases}
\end{equation}

If $\alpha>2$, then $A_3$ is trivially bounded by 
$$\nu_{\varphi_{c,N}}\left[B_L\right]\nu^{L-1}_{\varphi_{c,N}}\left[\sum_{x=1}^{L-1}\eta_x\geq(L-1)\rho_{c,N}+C_L\right].$$
Choose $C_L=\sqrt{Na_L}$. The second term has been proved in the proof of claim \eqref{claim2} 
smaller than or equal to $L\nu_{\varphi_c, N}[\eta_x^2]C_L^{-2}$ and the order of $nu_{\varphi_c, N}$ has been given in \eqref{den2}.

The desired bound of $A_3$ follows from the above estimates and the assumption on $\rho_L$  after a straightforward computation.
\end{proof}

\begin{lemma}\label{E2}
$$\lim_{L\to\infty}\frac{\nu^L_{\varphi_{c,N}}\left[E_2\right]}{L\nu_{\varphi_{c,N}}\left[N-(L-1)\rho_{c,N}\right]}\,=\,0.$$
\end{lemma}
\begin{proof}
Since  $\nu^L_{\varphi_{c,N}}$ is translation invariant under the operator $\sigma^{x,y}$ for $x,y\in \bb T_L$, $\nu^L_{\varphi_{c,N}}\left[E_2\right]$ is less than
\begin{equation}
\begin{split}
&L^2 \nu^L_{\varphi_{c,N}}\left[\sum_{x=1}^L\eta_x=N, \,\eta_L\geq B_L, \,\eta_{L-1}\geq B_L\right]\\
\leq\, &L^2 \sum_{k=0}^N \nu^{L-2}_{\varphi_{c,N}}\left[\sum_{x=1}^{L-2}\eta_x=N-k\right] \nu^2_{\varphi_{c,N}}\left[\eta_1\geq B_L,\, \eta_2\geq B_L,\, \eta_1+\eta_2=k\right]\\
\leq\, & L^2\sup_{0\leq k\leq N}\nu^2_{\varphi_{c,N}}\left[\eta_1\geq B_L,\, \eta_2\geq B_L,\, \eta_1+\eta_2=k\right]\\
\leq\, & L^2\left(\nu_{\varphi_{c,N}}\left[\eta_x= B_L\right]\right)^2.
\end{split}
\end{equation}
A direct computation shows that 
\begin{equation}
\frac{L^2\left(\nu_{\varphi_{c,N}}\left[\eta_x= B_L\right]\right)^2}{L\nu_{\varphi_{c,N}}\left[N-(L-1)\rho_{c,N}\right]}\,=\,
\begin{cases}
O(\frac{(\log N)^{4\alpha}}{N^{\alpha-1}}),  &\text{if} \quad \alpha>2\\
O(\frac{(\log N)^{2\alpha-1}}{N}), &\text{if} \quad \alpha\in [1,2]\\
\end{cases}
\end{equation}
which yields the lemma.
\end{proof}
\end{section}
\smallskip\noindent{\bf Acknowledgments.} We would like to thank Milton Jara for useful discussion about this problem, especially the suggestion on the proof of Theorem \ref{lclt}. The author also thanks the finanical support of FAPESP Grant No.2019/02226-2.

\end{document}